\newif\ifRACSAMsubmissionSNJNL
\RACSAMsubmissionSNJNLfalse

\ifRACSAMsubmissionSNJNL
\documentclass[pdflatex,sn-mathphys-num]{sn-jnl}
\hypersetup{%
        citecolor=blue!70!black,
        linkcolor=blue!70!black,
        urlcolor=blue!70!black,
}

\else
\documentclass{article}
\fi

\usepackage[utf8]{inputenc}
\usepackage[T1]{fontenc}

\usepackage{amsfonts}
\usepackage{amsthm}
\usepackage{amsmath}
\usepackage{mathrsfs}

\usepackage{enumerate}
\usepackage{graphicx}
\usepackage{hyperref}
\usepackage{xcolor}
\usepackage{bookmark}
\hypersetup{%
        bookmarksnumbered=true,
}

\usepackage{pgfplots}
\pgfplotsset{compat=1.15}
\usetikzlibrary{arrows}

\usepackage{soul}

\hypersetup{pdfborder={0 0 .2}}

\usepackage[russian,english]{babel}

\ifRACSAMsubmissionSNJNL
\else
\usepackage[nobreak]{cite}
\fi

\usepackage{mathtools}
\usepackage[capitalize,nameinlink]{cleveref}
\Crefname{enumi}{}{}
\Crefname{subsection}{Subsection}{Subsections}

\newcommand{\R}{\mathbb R}
\newcommand{\N}{\mathbb N}

\newcommand{\A}{\mathcal A_{c}}
\newcommand{\An}{\mathcal A_{nc}}
\newcommand{\F}{\mathcal F}
\newcommand{\ext}{\textnormal{ext}}
\newcommand{\chull}{\textnormal{c-hull}}
\DeclareMathOperator{\card}{card}
\newcommand{\cardOfInterval}{\card\,}

\newcommand{\lambdaone}{\lambda^1}
\newcommand{\chara}{\chi}

\newtheorem{theorem}{Theorem}[section]
\newtheorem{question}[theorem]{Question}
\newtheorem{lemma}[theorem]{Lemma}
\newtheorem{corollary}[theorem]{Corollary}

\newtheorem{remark}[theorem]{Remark}
\theoremstyle{definition}

\newtheorem{RAWexample}[theorem]{Example}
\newenvironment{example}{\begin{RAWexample}}{\leavevmode\unskip\penalty9999 \hbox{}\nobreak\hfill
                \hskip 0.8em\relax
                \eodsymbol
                \end{RAWexample}}
\newcommand{\eodsymbol}{\hbox{\rlap{$\ulcorner$}$\lrcorner$}}

\ifRACSAMsubmissionSNJNL

\title[Betweenness and concentric circles]{%
The betweenness relation
distinguishes
non-similar pairs of concentric circles}

\else
\title{%
The betweenness relation
distinguishes
non-similar pairs of concentric circles}

\newcommand\authoremail[1]{\\\texttt{\footnotesize #1}}
\author{Martin Dole\v{z}al\authoremail{dolezal@math.cas.cz}\and Jan Kol\'{a}\v{r}\authoremail{kolar@math.cas.cz}\and Janusz Morawiec\authoremail{janusz.morawiec@us.edu.pl}}
\fi

\date{\today}

\begin{document}

\ifRACSAMsubmissionSNJNL
\else

\maketitle

\begingroup 

\fi

\abstract{\mathversion{normal}
Two subsets $A, B$ of the plane are betweenness isomorphic if there is a bijection $f\colon A\to B$ such that, for every $x,y,z\in A$, the point $f(z)$ lies on the line segment connecting $f(x)$ and $f(y)$ if and only if $z$ lies on the line segment connecting $x$ and $y$.
In general, it is quite difficult to tell whether two given subsets of the plane are betweenness isomorphic.
We concentrate on the case when the sets $A,B$ belong to the family $ \mathcal A_c$
of unions of pairs of concentric circles
in the plane.
We prove that $A, B \in \mathcal A_c$ are betweenness isomorphic if and only if they are similar.
In particular, there are continuum many betweenness isomorphism classes in $ \mathcal A_c$, and each of these classes consists exactly of
all scaled translations of an arbitrary representative of the class. 
Furthermore, we show that every betweenness isomorphism between sets $A,B\in \mathcal A_c$ is exactly the restriction of a scaled isometry 
of the plane.}

\ifRACSAMsubmissionSNJNL
\else

\newcommand\keywords[1]{{\noindent\footnotesize\emph{Keywords}: #1.\par}}
\newcommand\pacs[2][PACS]{{\noindent\footnotesize\emph{#1}: #2\par}}
\vspace{4mm}
\endabstract 
\endgroup    

\fi

\subsubsection*{Acknowledgements}
{\small\footnotesize
Martin Dole\v{z}al:
Supported by the EXPRO project 20-31529X awarded by the Czech Science Foundation and by the Institute of Mathematics, Czech Academy of Sciences (RVO 67985840).\\
Jan Kol\'{a}\v{r}:
Supported by the Institute of Mathematics, Czech Academy of Sciences (RVO 67985840).\\
Janusz Morawiec:
The research activities co-financed by the funds granted under the Research Excellence Initiative of the University of Silesia in Katowice and by the University of Silesia Mathematics Department (Iterative Functional Equations and Real Analysis program).
\par
}

\section{Introduction}

In this paper, we continue
our research,
started in \cite{DKM2024},
aimed at
the classification and characterization of the (Euclidean) betweenness isomorphism classes
of sets in the plane; this time we focus on
a different family of sets.
Sets $A$, $B$ in the plane 
are said to be \emph{betweenness isomorphic} if there exists a bijection $f\colon A\to B$ such that, for all $x,y,z\in A$,
we have $f(x)\in[f(y),f(z)]$ if and only if $x\in[y,z]$, where $[u,v]$ denotes the line segment
\[
[u,v]=\big\{w\in\R^2:w=\lambda u+(1-\lambda)v\textnormal{ for some }\lambda\in[0,1]\big\},\quad u,v\in\R^2.
\]
Basic information on the ternary relation called betweenness, which is widely studied in connection with broadly understood geometry, can be found, e.g., in \cite{Soltan1984,Pambuccian2011}. 
The reader may also consult \cite{Vel1993}.

A common generalization of two problems posed by Wies\l aw Kubi\'{s} (see \cite[\allowbreak Problems~1.1 and 1.2]{DKM2024}) reads as follows.

\begin{question}\label{Question1}
Let $\F$ be a family consisting of certain subsets of the plane.
\begin{enumerate}
\item[{\rm (A)}] For which pairs $A,B\in\F$ does there exist a betweenness isomorphism $f\colon A\to B$?
\item[{\rm (B)}] How many betweenness isomorphism classes in $\F$ are there? How to characterize these classes?
\end{enumerate}
\end{question}

Note that Problem 3.29.3 from \cite{Vel1993} is in the spirit of \cref{Question1} since it concerns the classification of betweenness isomorphism classes of convex polytopes in Euclidean spaces.

Denote by $\A$ the family of all sets that are the union of two concentric circles in the plane and by $\mathcal B_l$ the family of all subsets of the plane of the form $C\cup F$, where $C$ is a circle and $F$ is a set of $l$ pairwise distinct points
enclosed by~$C$. 
Then \Cref{Question1} with $\F=\mathcal B_l$ reduces to~\cite[Problem 1.1]{DKM2024}, whereas with $\F=\A$ to~\cite[Problem 1.2]{DKM2024}. 

A difficulty in answering \Cref{Question1} 
is that tools developed for a given family $\F$ may be completely useless for another family. 
This is the case with the two families from Kubi\'{s}' problems.
Namely, in~\cite{DKM2024} we gave an answer to part (A) of \Cref{Question1}
for the family $\F=\mathcal B_l$ and a full answer to part (B)
for a special subfamily of $\F=\mathcal B_3$;
our main tool was a certain group action on the circle (using `reversions' about the points).
Unfortunately, the same tool may not be used in the case of the family 
$\F=\A$: the difference is that
the orbit (introduced in \cite[Section~2.3]{DKM2024}) 
of any point of the outer circle would be the whole outer circle, 
and determining the isomorphism class of the relevant
group action then
remains as the unresolved part
of the task.
Therefore, to provide an answer to~\cite[Problem 1.2]{DKM2024}
we need to develop some new tools.

Assuming that we know the answer to part (A) of \cref{Question1}, it is natural to
ask what do the betweenness isomorphisms $f\colon A\to B$ look like.
Unfortunately, it seems that not much is known.
From the results proven in~\cite{KubisMorawiecZurcher2022}, we immediately obtain at least some answers in the case when the sets $A,B$ are convex.
For example, by~\cite[Theorem~2.1]{KubisMorawiecZurcher2022}, if one (or, equivalently, both) of the sets $A,B$ is not contained in a single line, then every betweenness isomorphism $f\colon A\to B$ extends to a unique homography of the plane. 

For non-convex sets the situation can be
more complicated,
as can be seen in
\cite[Theorem 3.1]{DKM2024} for the case of the family $\mathcal B_l$.
For example,
let
$A_i=C_i\cup\{p_i\}\in\mathcal B_1$,
for $i=1,2$.
Let also
$E_i$ be
the intersection of the circle $C_i$ with an open half-plane determined by any line passing through the point $p_i$.
Then any bijection 
$h\colon E_1\to E_2$
has exactly two distinct extensions to a betweenness isomorphism $f\colon A_1\to A_2
$. 
A very simple example showing that a betweenness isomorphism of two non-convex sets $A,B\in\F$ can be very far from a "regular" function 
(as in the case above) is the family $\F=\mathcal B_0$ 
consisting of all circles in the plain; 
indeed, any bijection between two circles is a betweenness isomorphism.

Assume again that we know the answer to part (A) of \cref{Question1}. 
Clearly, every betweenness isomorphism $f\colon A\to B$ keeps the collinearity of points, i.e.\ for every $a,b,c\in A$, the points $a,b,c$ are collinear if and only if the points $f(a),f(b),f(c)$ are collinear. 
This exactly means that every betweenness isomorphism is also a \emph{collinearity isomorphism}. Therefore, we can ask for which families $\F$ both the isomorphisms coincide. 
It is not difficult to see that there are families $\F$ for which betweenness isomorphisms and collinearity isomorphisms cannot
coincide
(see \cref{exBC1} for the family $\F$ consisting of five points sets, and \cref{exBC2} for the family consisting of subsets of a line).

In the context of the results obtained, it is natural to ask for an answer to \cref{Question1} in the case where the family $\F$ consists of all sets that are the union of two non-concentric circles in the plane. 
Let us denote this family by $\An$.
We do not yet have a comprehensive answer to \cref{Question1} for $\F=\An$. However, some basic facts about the family $\An$ can be formulated easily.

Our paper concerns the family $\A$, and is organized as follows. 

In \cref{sec:P} we introduce notions and formulate basic facts needed in the next sections.

\cref{sec:results} aims to provide a complete answer to Kubi\'{s}' Problem~1.2 from \cite{DKM2024}, i.e., to give an answer to \cref{Question1} with $\F=\A$, by proving that two sets from the family $\A$ are betweenness isomorphic if and only if they have the same ratio of the radii of the two concentric circles (see \Cref{cor:betwIso}). 
In particular, we obtain that there are continuum many betweenness isomorphism classes in $\A$, and each of these classes consists exactly of all scaled translations of an arbitrary representant of the class. 
While not too much surprising, these results are in sharp contrast with the situation considered in \cite{DKM2024} in the case where $\F=\mathcal B_1$, or $\F=\mathcal B_2$, or
$\F=\mathcal B_3^{\textrm{col}} \subset \mathcal B_3$ where the three points are required to be collinear.

In \cref{sec:SIB}, we describe all betweenness isomorphisms between sets of the considered family $\A$ belonging to the same betweenness isomorphism class. 
It turns out that these are exactly the restrictions of those
scaled isometries
that map $A$ onto $B$ (see \Cref{thm:extension}). 
Recall that a \emph{scaled isometry}
(or, a \emph{similarity}) is a map
$\Phi\colon\R^2\to\R^2$ for which there exists $C>0$ such that
\begin{equation}\label{eq:scaledIso}
\|\Phi(x)-\Phi(y)\|=C\|x-y\|,\quad x,y\in\R^2.
\end{equation}
Again we see that the obtained result is in sharp contrast with the case of the family $\mathcal B_l$ considered in \cite{DKM2024}.

The goal of \cref{sec:BC} is to show that any collinearity isomorphism between two sets of the family $\A$ is also a betweenness isomorphism. Therefore, both
isomorphism notions coincide on sets of the family $\A$. 

Finally, in \cref{sec:AAn}, we give an example of a set $S\in\A$ that is not betweenness isomorphic with any set $R\in\An$ and vice versa.


\section{Preliminaries}\label{sec:P}
Throughout this paper, the symbol $\A$ is reserved for the family consisting of those sets $A\subseteq\R^2$ which are the union of two distinct concentric circles.

We denote the circle with center $c\in\R^2$ and radius $\rho\in(0,\infty)$
by $S(c,\rho)$.
We say that $y\in\R^2$ \emph{lies between $x\in\R^2$ and $z\in\R^2$} if $y=\lambda x+(1-\lambda)z$ for some $\lambda\in[0,1]$.
For $S\subseteq\R^2$ and $x,z\in S$, we put
\[
[x,z]_S=\{y\in S:y\text{ lies between $x$ and $z$}\},
\]
and
\begin{equation}\label{eq:openInterval}
(x,z)_S=\{y\in S\setminus\{x,z\}:y\text{ lies between $x$ and $z$}\}.
\end{equation}

We recall the notions of extreme points and collinear hulls from~\cite{DKM2024}.
Let $S\subseteq\R^2$.
A point $y\in S$ is said to be \emph{extreme in~$S$} if there are no $x,z\in S\setminus\{y\}$ such that $y$ lies between $x$ and $z$.
The set of all extreme points in $S$ is denoted by $\ext(S)$.
By~\cite[Lemma~2.1(i)]{DKM2024}, all betweenness isomorphisms preserve extreme points. 
A set $A\subseteq S$ is called \emph{collinearly closed in $S$} if $z\in A$ whenever $z\in S$ and there exist distinct $x,y\in A$ such that $x,y,z$ are collinear.
The \emph{collinear hull} $\chull_S(A)$ of an arbitrary set $A\subseteq S$ in $S$ is the smallest set which is collinearly closed in $S$ and which contains $A$.
By~\cite[Lemma~2.1(ii)]{DKM2024}, every betweenness isomorphism maps the collinear hull of a given set onto the collinear hull of its image.

Given a set $A$, we denote its cardinality by $\card A$.

We begin with the following lemma that will be useful later.

\begin{lemma}
\label{lem:identity}
Let\/ $S\subseteq\R^2$ be a set which intersects each line in finitely many points only.
Let $g\colon S\to S$ be a betweenness isomorphism of\/ $S$ onto itself.
Then the set\/ $W=\{s\in S:g(s)=s\}$ is collinearly closed in $S$.
\end{lemma}

\begin{proof}
Suppose that $x,y\in W$ are distinct.
Suppose also that $z\in S$ is such that $x,y,z$ are collinear; we must show that $z\in W$.
We have
\[
g(\chull_S(\{x,y\}))\stackrel{\text{\cite[Lemma~2.1(ii)]{DKM2024}}}{=}\chull_S(\{g(x),g(y)\})=\chull_S(\{x,y\}),
\]
which means that the restriction of $g$ to $\chull(\{x,y\})$ is a betweenness isomorphism of $\chull(\{x,y\})$ onto itself.
As the set $\chull_S(\{x,y\})$ is contained in a single line, it is finite by our assumption on $S$.
This, together with the fact that $g$ fixes both $x$ and $y$, easily implies that $g$ must be the identity on $\chull_S(\{x,y\})$.
In particular, $g(z)=z$ as we needed.
\end{proof}

Let $S=S(c,\rho)\cup S(c,\rho')$ where $c\in\R^2$ and $\rho<\rho'$.
For every $\alpha\in\R$, we define
\[
A_S(\alpha)=\Bigl\{c+\rho'(\cos\theta,\sin\theta):\theta\in\big(\alpha-\arccos\frac\rho{\rho'},\alpha+\arccos\frac\rho{\rho'}\big)\Bigr\},
\]
cf.~\Cref{fig:imageTangent}.
So each $A_S(\alpha)$ is an arc of the circle $S(c,\rho')$ of length $2\rho'\arccos\frac\rho{\rho'}$.
Moreover, straightforward calculations show that the line passing through the two endpoints of $A_S(\alpha)$ is a tangent line of the circle $S(c,\rho)$.

\begin{lemma}
\label{lem:arcs}
Let $S,R\in\A$ and suppose that $f\colon S\to R$ is a betweenness isomorphism.
Then for every $\alpha\in\R$ there is $\beta\in\R$ such that $f(A_S(\alpha))=A_R(\beta)$,
and such that $f$ maps each of the two endpoints of the arc $A_S(\alpha)$ onto an endpoint of the arc $A_R(\beta)$.
\end{lemma}

\begin{proof}
Suppose that $S=S(c,\rho)\cup S(c,\rho')$,
where $c\in\R^2$ and $\rho<\rho'$, $R=S(d,\tau)\cup S(d,\tau')$
with $d\in\R^2$ and $\tau<\tau'$, and $\alpha\in\R$.
By~\cite[Lemma~2.1(i)]{DKM2024}, we have $f(\ext(S))=\ext(R)$.
Obviously, it holds that $\ext(S)= S(c,\rho')$ and $\ext(R)=S(d,\tau')$.
As $f$ is a bijection, it follows that $f(S(c,\rho))=S(d,\tau)$.
In particular,
\[
f(c+\rho(\cos\alpha,\sin\alpha))\in S(d,\tau).
\]
So there is $\beta\in\R$ (which is unique modulo $2\pi$) such that
\begin{equation}
\label{eq:u_Su_R}
f(c+\rho(\cos\alpha,\sin\alpha))=d+\tau(\cos\beta,\sin\beta).
\end{equation}

We put
\begin{equation*}
\begin{aligned}
u_S&=c+\rho(\cos\alpha,\sin\alpha),\\
v_S&=c+\rho'\Big(\cos\big(\alpha-\arccos\frac\rho{\rho'}\big),\sin\big(\alpha-\arccos\frac\rho{\rho'}\big)\Big),\\
w_S&=c+\rho'\Big(\cos\big(\alpha+\arccos\frac\rho{\rho'}\big),\sin\big(\alpha+\arccos\frac\rho{\rho'}\big)\Big),
\end{aligned}
\end{equation*}
see \Cref{fig:imageTangent}.
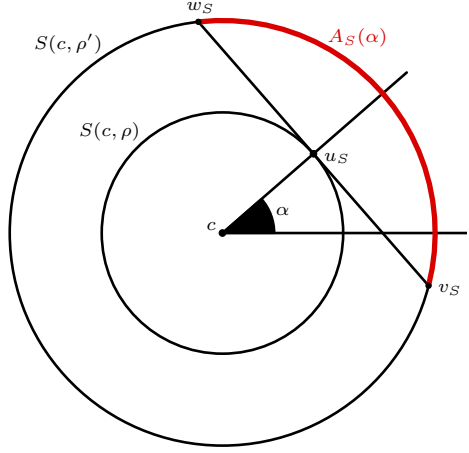
\begin{figure}
\centering
\colorlet{myredA}{red!85!black}
\begin{tikzpicture}[line cap=round,line join=round,>=triangle 45,x=1cm,y=1cm,scale=0.5]
\clip(-7,-6.2) rectangle (5.5,7);
\draw [shift={(-1,0)},line width=0.8pt,fill=black,fill opacity=0.1] (0,0) -- (0:1.382) arc (0:41.08:1.382) -- cycle;
\draw [line width=1pt] (-1,0) circle (5.656);
\draw [line width=1pt] (-1,0) circle (3.210);
\draw [line width=1pt,domain=-1.621:4.480] 
plot(\x,{(7.888-2.42*\x)/2.11});
\draw [line width=1pt,domain=-1:3.9] 
plot(\x,{(2.11+2.11*\x)/2.42});
\draw [line width=1pt,domain=-1:6.235] plot(\x,{0*\x});
\draw [shift={(-1,0)},line width=2pt,color=myredA]  plot[domain=-0.250:1.684,variable=\t]({1*5.656*cos(\t r)+0*5.656*sin(\t r)},{0*5.656*cos(\t r)+1*5.656*sin(\t r)});
\begin{scriptsize}
\draw [color=myredA](1.6,5.65) node[anchor=north west] {$A_S(\alpha)$};
\draw [fill=black] (-1,0) circle (2.5pt);
\draw[color=black] (-1.3,0.2) node {$c$};
\draw[color=black] (-5.1,5.05) node {$S(c,\rho')$};
\draw [fill=black] (1.42,2.11) circle (2.5pt);
\draw[color=black] (2.05,2) node {$u_S$};
\draw[color=black] (-4,2.7) node {$S(c,\rho)$};
\draw[color=black] (0.6,0.58) node {$\alpha$};
\draw [fill=black] (-1.640,5.620) circle (2pt);
\draw[color=black] (-1.58,6) node {$w_S$};
\draw [fill=black] (4.480,-1.4004) circle (2pt);
\draw[color=black] (5.05,-1.5) node {$v_S$};
\end{scriptsize}
\end{tikzpicture}
\caption{Position of the arc $A_S(\alpha)$ and the points $u_S,v_S,w_S$ in $S$.}
\label{fig:imageTangent}
\end{figure}
Then $v_S,w_S$ are the endpoints of the arc $A_S(\alpha)$.
Similarly, we put
\begin{equation*}
\begin{aligned}
u_R&=d+\tau(\cos\beta,\sin\beta),\\
v_R&=d+\tau'\Big(\cos\big(\beta-\arccos\frac\tau{\tau'}\big),\sin\big(\beta-\arccos\frac\tau{\tau'}\big)\Big),\\
w_R&=d+\tau'\Big(\cos\big(\beta+\arccos\frac\tau{\tau'}\big),\sin\big(\beta+\arccos\frac\tau{\tau'}\big)\Big);
\end{aligned}
\end{equation*}
then $v_R,w_R$ are the endpoints of the arc $A_R(\beta)$.

By~\eqref{eq:u_Su_R}, we have $f(u_S)=u_R$.
Further, for every $s\in S$, we have
\begin{equation}
\label{eq:cardinalities}
\begin{split}
\card\chull_R(\{f(s),u_R\}
=&\card\chull_R(\{f(s),f(u_S)\}\\
\stackrel{\textnormal{\cite[Lemma~2.1(ii)]{DKM2024}}}{=}{}&\card\chull_S(\{s,u_S\})
=\card\chull_S(\{s,u_S\}).
\end{split}
\end{equation}
It is easy to see from the picture that
\begin{equation*}
\{v_S,w_S\}=\{s\in S:\card\chull_S(\{s,u_S\})=3\}
\end{equation*}
and similarly
\begin{equation*}
\{v_R,w_R\}=\{r\in R:\card\chull_R(\{r,u_R\})=3\}.
\end{equation*}
Combining this observation with~\eqref{eq:cardinalities}, it follows that
\[
f(\{v_S,w_S\})=\{v_R,w_R\},
\]
which proves the assertion concerning the endpoints.
Further, one can easily check that
\begin{equation}
\label{eq:A_S}
A_S(\alpha)=\{
s\in\ext(S) :
\card\chull_S(\{s,u_S\})=4\text{ and }(s,u_S)_S=\emptyset\},
\end{equation}
and similarly
\begin{equation}
\label{eq:A_R}
A_R(\beta)=\{
r\in\ext(R) :
\card\chull_R(\{r,u_R\})=4\text{ and }(r,u_R)_R=\emptyset\}.
\end{equation}
As $f(\ext(S))=\ext(R)$ and $f(u_S)=u_R$, we have by~\eqref{eq:A_R} that
\begin{equation}
\label{eq:A_R2}
\begin{split}
A_R(\beta)=\{f(s):\quad &s\in\ext(S)
\text{ and }\card\chull_R(\{f(s),u_R\})=4\\
&\text{and }(f(s),u_R)_R =\emptyset\}.
\end{split}
\end{equation}
By comparing~\eqref{eq:A_S} and~\eqref{eq:A_R2}, 
we see that in order
to conclude that $f(A_S(\alpha))=A_R(\beta)$, it only remains to
apply equation~\eqref{eq:cardinalities} and the obvious fact that
\[
(s_1,s_2)_S=\emptyset\iff(f(s_1),f(s_2))_R=\emptyset,\quad s_1,s_2\in S.
\qedhere
\]
\end{proof}

In~\Cref{sec:results}, we define an invariant $M(\cdot)$ of sets from the collection $\A$ with respect to betweenness isomorphisms.
This invariant will be used in~\Cref{cor:betwIso} to determine which pairs of sets from $\A$ are betweenness isomorphic.
But before doing so, we introduce a simpler invariant which will hopefully provide some insight to the ideas described later.

Let $S\in\A$.
We define $m(S)$ to be the smallest natural number $n$ such that there exist $\alpha_1,\ldots,\alpha_n\in\R$ with
\[
\ext(S)=\bigcup_{i=1}^nA_S(\alpha_i).
\]

\begin{lemma}
\label{lem:invariantEasy}
$m(\cdot)$ is an invariant with respect to betweenness isomorphisms.
\end{lemma}

\begin{proof}
Suppose that $S,R\in\A$ are betweenness isomorphic, i.e., there exists a betweenness isomorphism $f\colon S\to R$.
Suppose that $\alpha_1,\ldots,\alpha_n\in\R$ are such that $\ext(S)=\bigcup_{i=1}^nA_S(\alpha_i)$.
By~\Cref{lem:arcs}, there are $\beta_1,\ldots,\beta_n\in\R$ such that
\[
f(A_S(\alpha_i))=A_R(\beta_i),\quad i=1,\ldots,n.
\]
Then
\[
\ext(R)=f(\ext(S))=f\Big(\bigcup_{i=1}^nA_S(\alpha_i)\Big)=\bigcup_{i=1}^nA_R(\beta_i).
\]
By the definition of $m(\cdot)$, it follows that $m(R)\le m(S)$.

As the situation is completely symmetric, we must also have $m(S)\le m(R)$, and so $m(S)=m(R)$.
\end{proof}

\begin{example}
Let $S=S(0,\rho)\cup S(0,1)$ and $R=S(0,\tau)\cup S(0,1)$ where $\rho<\frac12$ and $\sqrt{\frac12}\le\tau<1$.
Then the sets $S$ and $R$ are not betweenness isomorphic.
\end{example}

\begin{proof}
Each of the arcs $A_S(\alpha)$, $\alpha\in\R$, has length
\[
2\arccos\rho>2\arccos\frac12=\frac23\pi,
\]
and so the set $\ext(S)=S(0,1)$ can be covered by three such arcs.
Consequently, $m(S)\le3$ (in fact, $m(S)=3$).

On the other hand, each of the arcs $A_R(\beta)$, $\beta\in\R$, has length
\[
2\arccos\tau\le2\arccos\sqrt{\frac12}=\frac12\pi,
\]
and we need at least $4$ such arcs to cover the set $\ext(R)=S(0,1)$.
So $m(R)\ge4$.

As $m(S)\neq m(R)$, we conclude by \Cref{lem:invariantEasy} that $S$ and $R$ cannot be betweenness isomorphic.
\end{proof}


\section{Answer to \texorpdfstring{\cref{Question1}}{Question \ref{Question1}} for
\(\F=\A\)}\label{sec:results}

For a natural number $k$, we say that a collection $\mathcal O$ of sets is a $k$-cover of a set $S$ if $S=\bigcup\mathcal O$ and every $s\in S$ belongs to at least $k$ elements of $\mathcal O$.

For $S\in\A$, we define $M(S)$ by
\begin{equation*}
\begin{split}
M(S)=\inf\Big\{\frac nk:\quad&n,k\in\N\text{ are such that there exist }\alpha_1,\ldots,\alpha_n\in\R\\
&\text{such that $\big\{A_S(\alpha_1),\ldots,A_S(\alpha_n)\big\}$ is a $k$-cover of $\ext(S)$}\Big\}.
\end{split}
\end{equation*}

\begin{lemma}
\label{lem:invariant}
$M(\cdot)$ is an invariant with respect to betweenness isomorphisms.
\end{lemma}

\begin{proof}
Suppose that $S,R\in\A$ are betweenness isomorphic, i.e., there exists a betweenness isomorphism $f\colon S\to R$.
Suppose that $\alpha_1,\ldots,\alpha_n\in\R$ are such that the collection $\{A_S(\alpha_1),\ldots,A_S(\alpha_n)\}$ is a $k$-cover of $\ext(S)$.
By~\Cref{lem:arcs}, there are $\beta_1,\ldots,\beta_n\in\R$ such that
\[
f(A_S(\alpha_i))=A_R(\beta_i),\quad i=1,\ldots,n.
\]
We fix $r\in\ext(R)$.
As $f$ is a bijection and $f(\ext(S))=\ext(R)$, we have $f^{-1}(r)\in\ext(S)$.
So there are at least $k$ distinct indices $i\in\{1,\ldots,n\}$ such that $f^{-1}(r)\in A_S(\alpha_i)$.
For each such index $i$, we then have $r\in A_R(\beta_i)$.
As this holds for every $r\in\ext(R)$, the collection $\{A_R(\beta_1),\ldots,A_R(\beta_n)\}$ is a $k$-cover of $\ext(R)$.
By the definition of $M(\cdot)$, it follows that $M(R)\le M(S)$.

As the situation is completely symmetric, we must also have $M(S)\le M(R)$, and so $M(S)=M(R)$.
\end{proof}

As we will see in the next lemma, for a given $S\in\A$, $M(S)$ equals the ratio of the circumference of the bigger circle and the length of each of the arcs $A_S(\alpha)$, $\alpha\in\R$.

\begin{lemma}
\label{lem:valueOfM}
Let $S=S(c,\rho)\cup S(c,\rho')$ where $c\in\R^2$ and $\rho<\rho'$.
Then
\[
M(S)=\frac{\pi}{\arccos\frac\rho{\rho'}}.
\]
\end{lemma}

\begin{proof}
Fix $0<\varepsilon<2\arccos\frac\rho{\rho'}$.
For every $i\in\N$, we define
\[
\alpha_i=i\big(2\arccos\frac\rho{\rho'}-\varepsilon\big).
\]
Then the intersection of the arcs $A_S(\alpha_i)$ and $A_S(\alpha_{i+1})$ is an arc of length $\rho'\varepsilon$, $i\in\N$.
Moreover, if $n,k\in\N$ are such that
\[
\alpha_n=n\big(2\arccos\frac\rho{\rho'}-\varepsilon\big)\ge2k\pi,
\]
then the collection $\{A_S(\alpha_1),\ldots,A_S(\alpha_n)\}$ is a $k$-cover of $S(c,\rho')=\ext(S)$.
This provides the estimate
\begin{equation*}
M(S)\le
\inf\Big\{\frac nk:\alpha_n\ge2k\pi \Big\}
=\inf\Big\{\frac nk\ge\frac{
2
\pi }{
2
 \arccos\frac\rho{\rho'}-\varepsilon}\Big\}
=\frac{
2
\pi }{
2
 \arccos\frac\rho{\rho'}-\varepsilon}.
\end{equation*}
Since this holds for every
$0<\varepsilon<2\arccos\frac\rho{\rho'}$,
we conclude that $M(S)\le\frac{\pi}{\arccos\frac\rho{\rho'}}$.

To prove the opposite inequality, suppose that $\alpha_1,\ldots,\alpha_n\in\R$ are such that the collection $\{A_S(\alpha_1),\ldots,A_S(\alpha_n)\}$ is a $k$-cover of $\ext(S)=S(c,\rho')$.
Then, if we denote by $\lambdaone$ the $1$-dimensional Lebesgue measure on the circle $S(c,\rho')$, we have
\begin{equation}
\begin{split}
2n\rho'\arccos\frac\rho{\rho'}=&\sum_{i=1}^n\lambdaone(A_S(\alpha_i))
=\sum_{i=1}^n\int_{S(c,\rho')}\chara_{A_S(\alpha_i)}\, d\lambdaone\\
=&\int_{S(c,\rho')}\sum_{i=1}^n\chara_{A_S(\alpha_i)}\,d\lambdaone
\ge k\lambdaone(S(c,\rho'))
=2k\pi\rho',
\end{split}
\end{equation}
and consequently
\[
\frac nk\ge\frac{\pi}{\arccos\frac\rho{\rho'}}.
\]
The inequality $M(S)\ge\frac{\pi}{\arccos\frac\rho{\rho'}}$ follows by taking the infimum over all choices of $\alpha_1,\ldots,\alpha_n$ as above.
\end{proof}

\begin{corollary}
\label{cor:betwIso}
Let $S=S(c,\rho)\cup S(c,\rho')$ and $R=S(d,\tau)\cup S(d,\tau')$ where $c,d\in\R^2$, $\rho<\rho'$ and $\tau<\tau'$.
Then the sets $S$, $R$ are betweenness isomorphic if and only if
\begin{equation}
\label{eq:ratioEquality}
\frac\rho{\rho'}=\frac\tau{\tau'}.
\end{equation}
\end{corollary}

\begin{proof}
Suppose first that the sets $S,R$ are betweenness isomorphic.
By \Cref{lem:invariant}, we have $M(S)=M(R)$.
Then~\eqref{eq:ratioEquality} immediately follows by \Cref{lem:valueOfM}.

Now suppose that~\eqref{eq:ratioEquality} holds true.
For every $C>0$ and every $u\in\R^2$, let $M_C\colon\R^2\to\R^2$ and $T_u\colon\R^2\to\R^2$ be the maps given by
\begin{equation}
\label{def:rescalingTranslation}
M_C(x)=Cx\quad\text{and}\quad T_u(x)=x+u,\quad x\in\R^2.
\end{equation}
Obviously, all these maps are betweenness isomorphisms of $\R^2$ onto $\R^2$.
We put
\[
f=T_d\circ M_{\frac{\tau'}{\rho'}}\circ T_{-c}.
\]
Then $f$, being a composition of betweenness isomorphisms, is also a betweenness isomorphism (of $\R^2$ onto $\R^2$).
Moreover, we have $f(S(c,\rho'))=S(d,\tau')$ and
\[
f(S(c,\rho))=S\Big(d,\frac{\rho\tau'}{\rho'}\Big)\stackrel{\eqref{eq:ratioEquality}}{=}S(d,\tau).
\]
So $f(S)=R$, and the restriction of $f$ to $S$ is a betweenness isomorphism from $S$ onto $R$.
\end{proof}


\section{The betweenness isomorphisms in \(\A\)}\label{sec:SIB}

By the proof of 
\cref{cor:betwIso},
whenever two sets from $\A$ are betweenness isomorphic, the betweenness isomorphism between them can be chosen as (the restriction of) the composition of two translations and a 
rescaling.
There exist more betweenness isomorphisms of sets from $\A$.
Indeed, fix $\rho'>\rho$ and put $S=S((0,0),\rho)\cup S((0,0),\rho')\in\A$.
Then any rotation (restricted to $S$) centered at the origin is a betweenness isomorphism of $S$ onto itself.
Similarly, any reflection (restricted to $S$) through a line passing through the origin is a betweenness isomorphism of $S$ onto itself.

Note that all the transformations of the plane mentioned above (i.e., translations, 
rescalings, rotations and reflections) are special cases of a
scaled isometry of the plane (that is, of a map $\Phi\colon\R^2\to\R^2$ for which there is $C>0$ with~\eqref{eq:scaledIso}).
In fact, every scaled isometry can be expressed as
a composition of an isometry and a 
rescaling (see \cite[5.51 and 5.61]{Coxeter1969}), 
and every isometry is well known to be a composition of a translation, a rotation and, possibly, of
a reflection (see \cite[3.51]{Coxeter1969}).
We also note that each
scaled isometry of the plane is a betweenness isomorphism of the plane onto itself, because the Euclidean betweenness (as a ternary relation) can be characterized using distances in an obvious way.

In the next theorem, we will show that every betweenness isomorphism $f$ between sets from $\A$ extends to a unique
scaled isometry of the plane.

\begin{theorem}
\label{thm:extension}
Let $S,R\in\A$ and suppose that $f\colon S\to R$ is a betweenness isomorphism.
Then $f$ extends to a unique
scaled isometry $\Phi\colon\R^2\to\R^2$.
\end{theorem}

\begin{proof}
The uniqueness is very easy, and so we prove only the existence.

Suppose that $S=S(c,\rho)\cup S(c,\rho')$ and $R=S(d,\tau)\cup S(d,\tau')$ where $c,d\in\R^2$, $\rho<\rho'$ and $\tau<\tau'$.
By \Cref{cor:betwIso}, we have
\[
r:=\frac\rho{\rho'}=\frac\tau{\tau'}.
\]
For every $C>0$ and every $u\in\R^2$, let $M_C\colon\R^2\to\R^2$ and $T_u\colon\R^2\to\R^2$ be the maps given by~\eqref{def:rescalingTranslation}.
Then $M_{\frac1{\rho'}}\circ T_{-c}$, resp. $M_{\frac1{\tau'}}\circ T_{-d}$, is a
scaled isometry (in particular, a betweenness isomorphism) of the plane which maps $S$, resp. $R$, onto $S(0,r)\cup S(0,1)$.
So the map
\begin{equation}
\label{def:h}
h:=(M_{\frac1{\tau'}}\circ T_{-d})\circ f\circ\big((M_{\frac1{\rho'}}\circ T_{-c})^{-1}|_{S(0,r)\cup S(0,1)}\big)
\end{equation}
is a betweenness isomorphism of $S((0,0),r)\cup S((0,0),1)$ onto itself.
If we prove that $h$ extends to a
scaled isometry $H$ of the plane then, by~\eqref{def:h}, $f$ extends to the 
scaled isometry
\[
(M_{\frac1{\tau'}}\circ T_{-d})^{-1}\circ H\circ(M_{\frac1{\rho'}}\circ T_{-c})
\]
of the plane.
Thus, without loss of generality, we can assume that $S=R=S((0,0),r)\cup S((0,0),1)$ for some $r\in(0,1)$.

By~\cite[Lemma~2.1(i)]{DKM2024}, $f$ preserves extreme points, that is, $f$ maps $S((0,0),1)$ bijectively onto itself.
We will show that the restriction of $f$ to $S((0,0),1)$ is continuous.
By \Cref{lem:arcs}, $f$ maps each $A_S(\alpha)$ (which is an open arc of the circle $S((0,0),1)$ of length $2\arccos r$), $\alpha\in\R$, onto some $A_R(\beta)$, $\beta\in\R$.
Every open arc of the circle $S((0,0),1)$ of length at most $2\arccos r$ equals $A_S(\alpha_1)\cap A_S(\alpha_2)$ for a suitable choice of $\alpha_1,\alpha_2\in\R$.
So $f$ maps every such arc onto $A_R(\beta_1)\cap A_R(\beta_2)$ for some $\beta_1,\beta_2\in\R$.
In particular, all such arcs are mapped by $f$ onto open sets.
As the collection of all open arcs of the circle $S((0,0),1)$ of length at most $2\arccos r$ forms a base of the topology of the circle $S((0,0),1)$, we obtain that the restriction of $f^{-1}$ to $S((0,0),1)$ is continuous.
But $f^{-1}$ maps the compact set $S((0,0),1)$ bijectively onto itself, and so $f=(f^{-1})^{-1}$ restricted to $S((0,0),1)$ is continuous, as well.

Next, we will show that $f$ is continuous.
As we already know that $f$ is continuous on $S((0,0),1)$, it only remains to verify continuity on $S((0,0),r)$.
We put
\[
U=\{(u,v)\in (S((0,0),1))^2:\cardOfInterval (u,v)_R=2\},
\]
where $(u,v)_R$ is as in~\eqref{eq:openInterval}.
That is, $U$ is the set of all pairs $(u,v)\in (S((0,0),1))^2$ such that there are two distinct points in $S((0,0),r)$ lying strictly between $u$ and $v$.
For every $(u,v)\in U$, let $\psi(u,v)\in S((0,0),r)$ be the unique point such that $\psi(u,v)\in(u,v)_R$ and $(u,\psi(u,v))_R=\emptyset$.
That is, $\psi(u,v)$ lies strictly between $u$ and $v$ and, from the two elements of $R$ with this property, it is the one lying closer to $u$.
The map $\psi\colon U\to R$ is obviously continuous.
For the rest of the proof, we apply the notation
\[
P(\alpha):=(\cos\alpha,\sin\alpha),\quad\alpha\in\R.
\]
Let $\theta\colon S(0,r)\to U$ be the map given by
\[
\theta(rP(\alpha))=(f(P(\alpha)),f(P(\alpha+\pi))),\quad\alpha\in\R.
\]
The fact that the range of $\theta$ is a subset of $U$ is given by the facts that
\[
\cardOfInterval (P(\alpha),P(\alpha+\pi))_S=2
\]
and that $f\colon S\to R$ is a betweenness isomorphism.
By the already proven continuity of $f$ on $S((0,0),1)$, the map $\theta$ is also continuous, as well as the composition $\psi\circ\theta\colon S((0,0),r)\to R$.
So it remains to show that $f$ coincides with $\psi\circ\theta$ on $S((0,0),r)$.
To this end, we fix $s\in S((0,0),r)$.
Let $\alpha\in\R$ be such that $s=rP(\alpha)$.
Then $s$ is the unique element of $S$ belonging to $(P(\alpha),P(\alpha+\pi))_S$ such that $(P(\alpha),s)_S=\emptyset$.
As $f\colon S\to R$ is a betweenness isomorphism, $f(s)$ must be the unique element of $R$ belonging to $(f(P(\alpha)),f(P(\alpha+\pi)))_R$ such that $(f(P(\alpha)),f(s))_R=\emptyset$.
By the definition of the maps $\psi,\theta$, we immediately obtain that $f(s)=\psi\circ\theta(s)$.
As $s\in S((0,0),r)$ was chosen arbitrarily, the continuity of $f$ follows.

Let $a,b$ be the endpoints of the arc $A_S(0)$.
By \Cref{lem:arcs}, there is $\beta\in\R$ such that $f$ maps $a,b$ to the endpoints of the arc $A_R(\beta)$.
Let $\Phi^{\pm}\colon\R^2\to\R^2$ be the map given by
\begin{equation}
\label{def:Phi}
\Phi^{\pm}\colon tP(\alpha)\mapsto tP(\beta\pm\alpha),\quad t\ge 0,\alpha\in\R.
\end{equation}
Then $\Phi^+$ is a rotation and $\Phi^-$ is a composition of a rotation and a reflection.
So, both $\Phi^+$ and $\Phi^-$ are isometries of the plane onto itself and, in particular, betweenness isomorphisms of the plane onto itself.
As the arcs $A_S(0),A_R(\beta)$ have the same length (equal to $2\arccos r$),
we see that, for some choice of the sign $\pm$,
the restrictions of $f$ and $\Phi^{\pm}$ to the two-element set $\{a,b\}$
coincide;
let us denote the corresponding choice of $\Phi^+$ or $\Phi^-$ shortly by $\Phi$.
If we show that $\Phi^{-1}\circ f$ extends to a 
scaled isometry $\widetilde H$ of the plane then $f$ extends to the
scaled isometry $\Phi\circ\widetilde H$ of the plane.

So, without loss of generality, we can assume that $f$ fixes both $a$ and $b$ (because $\Phi^{-1}\circ f$ does).
Under this additional assumption, we will in fact prove that $f$ is the identity on $S$, and so it extends to the identity of the plane, which is a particular case of a
scaled isometry of the plane onto itself.

We put
\begin{equation}
\label{def:W}
W=\{s\in S:f(s)=s\}.
\end{equation}
The set $W$ is non-empty, as $a,b\in W$.
Since the intersection of $S$ with any line consists of at most four points, we may apply \Cref{lem:identity} to obtain that the set $W$ is collinearly closed in $S$.
Further, as $f$ is continuous, the set $W$ is topologically closed in $S$.

Our next goal is to show that the set $W\cap S((0,0),1)$ is topologically dense in $S((0,0),1)$.
Let $\alpha\in\R$ and $\gamma\in\R$ be such that $f(P(\alpha))=P(\gamma)$.
We will verify that
\begin{equation}
\label{eq:kArcs}
f(P(\alpha+2k\arccos r))=P(\gamma\pm 2k\arccos r),\quad k\in\N
\cup \{0\},
\end{equation}
where the sign $\pm$ does not depend on $k$.
(However, we do not claim at this moment that the sign would be independent of $\alpha$.)
For $k=0$, \eqref{eq:kArcs} is just the equality $f(P(\alpha))=P(\gamma)$.
Let
$k=1$.
Then $P(\alpha)$ and $P(\alpha+2\arccos r)$ are endpoints of the arc $A_S(\alpha+\arccos r)$, and so it is enough to apply \Cref{lem:arcs}.
Indeed,
the lemma implies that 
$f(A_S(\alpha+\arccos r))$
equals $A_R(\beta)$
for some $\beta$
and that one of the endpoints
of the arc $A_R(\beta)$
is the point $f(P(\alpha))=P(\gamma)$.
The other endpoint
necessarily
is
$P(\gamma+2\arccos r)$ 
or
$P(\gamma-2\arccos r)$;
this also determines the sign.
Now suppose that $k\in\N$ is such that the assertion is true for every $i=
0,
1
,\ldots,k$.
We will assume that
\[
f(P(\alpha+2i\arccos r))=P(\gamma+2i\arccos r),\quad i=
0,
1
,\ldots,k;
\]
the other case (with the sign `$-$' instead of `$+$') is completely analogous.
Then the induction hypothesis, together with another application of \Cref{lem:arcs} (on the arc with endpoints $P(\alpha+2k\arccos r)$ and $P(\alpha+2(k+1)\arccos r)$), tell us that
\[
f(P(\alpha+2(k+1)\arccos r))=P(\gamma+2k\arccos r\pm2\arccos r).
\]
Now it only suffices to observe that, since $f$ is a bijection and since $0<4\arccos r<2\pi$, we have
\begin{equation*}
\begin{split}
P(\gamma+2k\arccos r-2\arccos r)&=f(P(\alpha+2(k-1)\arccos r))\\
&\neq f(P(\alpha+2(k+1)\arccos r)),
\end{split}
\end{equation*}
and so 
\[
f(P(\alpha+2(k+1)\arccos r))=P(\gamma+2(k+1)\arccos r),
\]
as we needed.
This concludes the proof of~\eqref{eq:kArcs}.

As the two points $a,b$ lie in $W$
and are of the form $P(\pm \arccos r)$,
we have~\eqref{eq:kArcs}
for $\alpha=\gamma=-\arccos r$
and with `plus' sign;
consequently
\[
f(P((2k-1)\arccos r)))=P((2k-1)\arccos r)),\quad k=0,1,2,\ldots.
\]
So all the points $P((2k-1)\arccos r)$, $k=0,1,2,\ldots$, belong to $W$.
If $\frac{2\arccos r}{2\pi}$ is an irrational number then these  points form a dense subset of $S((0,0),1)$, and so $W\cap S((0,0),1)$ is also a dense subset of $S((0,0),1)$.
From now on, we assume that $\frac{2\arccos r}{2\pi}$ is a rational number.

Let $j$ be the smallest natural number such that $j\arccos r$ is a multiple of $\pi$.
We will show that if $\alpha\in\R$ is such that $P(\alpha)\in W$ then also $P(\alpha+\pi)\in W$.
This is easier if $j$ is an even number.
Indeed, then $j\arccos r$ must be an odd multiple of $\pi$ (otherwise, $\tfrac j2\arccos r$ is still a multiple of $\pi$ and $\tfrac j2<j$, a contradiction), and so
\[
P(\alpha+\pi)=P(\alpha+j\arccos r)=P(\alpha-j\arccos r).
\]
As $j$ is even and $f(P(\alpha))=P(\alpha)$, it follows by~\eqref{eq:kArcs}
        used with $\gamma=\alpha$
that
\[
f(P(\alpha+\pi))=f(P(\alpha+j\arccos r))=P(\alpha\pm j\arccos r)=P(\alpha+\pi).
\]
So $P(\alpha+\pi)\in W$, as we wanted.
Now suppose that $j$ is an odd number and let $l\in\N$ be such that $j\arccos r=l\pi$.
Then, as $j+1$ is even, \eqref{eq:kArcs} implies that the two element set consisting of the points
\[
A^+:=P(\alpha+l\pi+\arccos r)=P(\alpha+(j+1)\arccos r)
\]
and
\[
A^-:=P(\alpha+l\pi-\arccos r)=P(\alpha-(j+1)\arccos r)
\]
is mapped by $f$ onto itself.
As $rP(\alpha+l\pi)$ is the unique element of $S$ lying strictly between $A^+$ and $A^-$ and as $f$ is a betweenness isomorphism, it follows that $rP(\alpha+l\pi)$ is a fixed point of $f$, that is, $rP(\alpha+l\pi)\in W$.
The points $P(\alpha)$, $rP(\alpha+l\pi)$ and $P(\alpha+\pi)$ are collinear, and we already know that the first two of them belong to $W$.
Since $W$ is collinearly closed in $S$, $P(\alpha+\pi)$ must belong to $W$ as well.
This completes the proof of the fact that $P(\alpha+\pi)\in W$ whenever $P(\alpha)\in W$.

Now it is easy to verify that if $\alpha\in\R$ is such that $P(\alpha)\in W$ then also $rP(\alpha)\in W$.
Indeed, if $P(\alpha)\in W$, then, as we already know, $P(\alpha+\pi)\in W$.
As the point $rP(\alpha)$ is collinear with $P(\alpha)$ and $P(\alpha+\pi)$ and the set $W$ is collinearly closed in $S$, the assertion follows.

Now we are ready to show the density of $W\cap S((0,0),1)$ in $S((0,0),1)$.
Suppose, for a contradiction, that $W\cap S((0,0),1)$ is not dense in $S((0,0),1)$.
Let $A\subseteq S((0,0),1)$ be a maximal open arc which does not intersect $W$.
Since $a,b\in W$,
\eqref{eq:kArcs} 
used with $\alpha=\gamma=-\arccos r$
implies that
$P((2k-1)\arccos r)$ belongs to $W$
for all $k=0,1,2,\ldots$;
hence the length of the arc $A$ is at most $2\arccos r<\pi$.
So there are $\alpha\in\R$ and $\gamma\in(0,\pi)$ such that $P(\alpha)$ and $P(\alpha+\gamma)$ are endpoints of the arc $A$ and
\begin{equation}
\label{eq:arcNotIntersected}
A=\{P(\alpha+\omega):\omega\in(0,\gamma)\}.
\end{equation}
As $A$ is a maximal open arc not intersecting $W$ and the set $W$ is closed, both $P(\alpha)$ and $P(\alpha+\gamma)$ belong to $W$.
It is easy to see that there is some $\omega\in(0,\gamma)$ such that the points $P(\alpha+\pi),rP(\alpha+\gamma)$ and $P(\alpha+\omega)$ are collinear, see \Cref{imageDensity}.
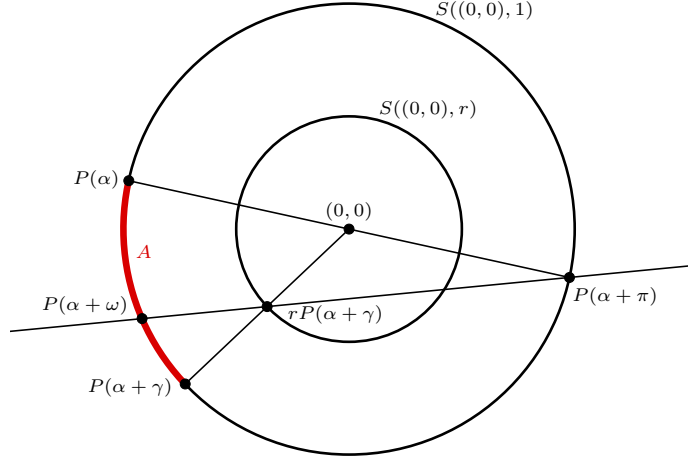
\begin{figure}
\centering
\colorlet{myred}{red!85!black}
\begin{tikzpicture}[line cap=round,line join=round,>=triangle 45,x=4cm,y=4cm,scale=0.75]
\clip(-1.5,-1.1) rectangle (1.5,1.1);
\draw [line width=1pt] (0,0) circle (1);
\draw [line width=1pt] (0,0) circle (0.5);
\draw [line width=0.6pt,domain=-1.8832:3.244] plot(\x,{(--0.4137-0.1292*\x)/-1.3397});
\draw [line width=0.6pt] (-0.9767,0.2146)-- (0.9767,-0.2146);
\draw [line width=0.6pt] (-0.726,-0.687)-- (0,0);
\draw [shift={(0,0)},line width=2.5pt,color=myred]  plot[domain=2.9253:3.8997,variable=\t]({1*1*cos(\t r)+0*1*sin(\t r)},{0*1*cos(\t r)+1*1*sin(\t r)});
\begin{scriptsize}
\draw [fill=black] (0,0) circle (2.5pt);
\draw[color=black] (0.0,0.08) node {$(0,0)$};
\draw[color=black] (0.6,0.97) node {$S((0,0),1)$};
\draw[color=black] (0.35,0.53) node {$S((0,0),r)$};
\draw [fill=black] (-0.976,0.2146) circle (2.5pt);
\draw[color=black] (-1.12,0.22) node {$P(\alpha)$};
\draw [fill=black] (-0.7260,-0.687) circle (2.5pt);
\draw[color=black] (-0.97,-0.7) node {$P(\alpha+\gamma)$};
\draw [fill=black] (-0.3630,-0.3438) circle (2.5pt);
\draw[color=black] (-0.06,-0.38) node {$rP(\alpha+\gamma)$};
\draw [fill=black] (0.9767,-0.2146) circle (2.5pt);
\draw[color=black] (1.18,-0.3) node {$P(\alpha+\pi)$};
\draw [fill=black] (-0.917,-0.397) circle (2.5pt);
\draw[color=black] (-1.17,-0.34) node {$P(\alpha+\omega)$};
\draw[color=myred] (-0.91,-0.1) node {$A$};
\end{scriptsize}
\end{tikzpicture}
\caption{The line passing through the points $P(\alpha+\pi)$ and $rP(\alpha+\gamma)$ intersects the arc $A$ at $P(\alpha+\omega)$ for some $\omega\in(0,\gamma)$.}
\label{imageDensity}
\end{figure}
But, by what we proved earlier, the points $P(\alpha+\pi)$ and $rP(\alpha+\gamma)$ belong to $W$.
Since $W$ is collinearly closed in $S$, the point $P(\alpha+\omega)$ belongs to $W$, too.
But, by~\eqref{eq:arcNotIntersected}, $P(\alpha+\omega)$ also belongs to $A$.
This is a contradiction with the assumption that the arc $A$ does not intersect $W$.
It follows that the set $W\cap S((0,0),1)$ is dense in $S((0,0),1)$.

We already know that, whenever $\alpha\in\R$ is such that $P(\alpha)\in W$, then also $rP(\alpha)\in W$.
Since $W\cap S((0,0),1)$ is dense in $S((0,0),1)$, we immediately obtain that $W\cap S((0,0),r)$ is also dense in $S((0,0),r)$.
Altogether, $W$ is dense in $S=S((0,0),r)\cup S((0,0),1)$.
As the set $W$ is also closed in $S$, we conclude that $W=S$.
So, by~\eqref{def:W}, $f$ is the identity on $S$, which completes the proof.
\end{proof}


\section{Collinearity vs. betweenness in \(\A\)}\label{sec:BC}

It follows directly from definitions that every betweenness isomorphism is a collinearity isomorphism.
However, the converse is not true in general, as we can see from the following easy examples.

\begin{example}\label{exBC1}
Let
\[
A=\{(-1,0),(0,-1),(0,0),(0,1),(1,0)\}
\]
and
\[
B=\{(0,-1),(0,0),(0,1),(1,0),(2,0)\}.
\]
Then the map $f\colon A\to B$ given by
\[
f(a)=
\begin{cases}
(2,0)&\textnormal{if }a=(-1,0)\\
a&\textnormal{if }a\neq(-1,0)
\end{cases}
\]
is a collinearity isomorphism of $A$ and $B$.
However, $A$ and $B$ are not betweenness isomorphic (because the sets do not have the same number of extreme points).
\end{example}

\begin{example}\label{exBC2}
Let $A=\{(x,0)\in\R^2:x\ge 0\}$ and $B=\{(x,0)\in\R^2:x\in\R\}$.
Then any bijection between $A$ and $B$ is a collinearity isomorphism of $A$ and $B$ but $A$ and $B$ are not betweenness isomorphic (because $A$ has an extreme point while $B$ does not).
\end{example}

Nevertheless, we will show that if $A,B$ belong to the class $\A$ (consisting of sets which are the union of two distinct concentric circles) then there is no difference between collinearity and betweenness isomorphisms from $A$ to $B$.
As a consequence, we can equivalently formulate our main results in terms of collinearity instead of betweenness.

We start by a variant of \Cref{lem:arcs}.
But this time, it seems to be easier to consider arcs of double length (compared to the arcs $A_S(\alpha)$, $\alpha\in\R$) first.

Let $S=S(c,\rho)\cup S(c,\rho')$ where $c\in\R^2$ and $\rho<\rho'$.
For every $\alpha\in\R$, we define
\begin{equation*}
\begin{split}
\widetilde A_S(\alpha)&=\Bigl\{c+\rho'(\cos\theta,\sin\theta):\theta\in\big(\alpha-2\arccos\frac\rho{\rho'},\alpha+2\arccos\frac\rho{\rho'}\big)\Bigr\}\\
&=A_S(\alpha-\arccos\frac\rho{\rho'})\cup\big\{c+\rho'(\cos\alpha,\sin\alpha)\big\}\cup A_S(\alpha+\arccos\frac\rho{\rho'}).
\end{split}
\end{equation*}

\begin{lemma}
\label{lem:doubleArcs}
Let $S=S(c,\rho)\cup S(c,\rho')$ and $R=S(d,\tau)\cup S(d,\tau')$ where $c,d\in\R^2$, $\rho<\rho'$ and $\tau<\tau'$.
Suppose that $f\colon S\to R$ is a collinearity isomorphism.
Then
\(
   f(S(c,\rho'))=S(d,\tau')
\).
Moreover,
for every $\alpha\in\R$ there is $\beta\in\R$ such that $f(\widetilde A_S(\alpha))=\widetilde A_R(\beta)$,
and such that $f$ maps each of the two endpoints of the arc $\widetilde A_S(\alpha)$ onto an endpoint of the arc $\widetilde A_R(\beta)$.
\end{lemma}

\begin{proof}
Suppose that $S=S(c,\rho)\cup S(c,\rho')$ where $c\in\R^2$ and $\rho<\rho'$, $R=S(d,\tau)\cup S(d,\tau')$ where $d\in\R^2$ and $\tau<\tau'$, and $\alpha\in\R$.
We have
\[
S(c,\rho')=\big\{s\in S:\exists s'\in S\quad\forall s''\in S\setminus\{s,s'\}\quad s,s',s''\textnormal{ are not collinear}\big\},
\]
and similarly
\[
S(d,\tau')=\big\{r\in R:\exists r'\in R\quad\forall r''\in R\setminus\{r,r'\}\quad r,r',r''\textnormal{ are not collinear}\big\}.
\]
As $f$ is a collinearity isomorphism, it follows that
\begin{equation}
\label{eq:outerCircles}
f(S(c,\rho'))=S(d,\tau').
\end{equation}

Let $\beta\in\R$ be such that
\begin{equation}
f(c+\rho'(\cos\alpha,\sin\alpha))=d+\tau'(\cos\beta,\sin\beta).
\end{equation}
We put
\begin{equation*}
\begin{aligned}
u_S&=c+\rho'(\cos\alpha,\sin\alpha),\\
v_S&=c+\rho'\Big(\cos\big(\alpha-2\arccos\frac\rho{\rho'}\big),\sin\big(\alpha-2\arccos\frac\rho{\rho'}\big)\Big),\\
w_S&=c+\rho'\Big(\cos\big(\alpha+2\arccos\frac\rho{\rho'}\big),\sin\big(\alpha+2\arccos\frac\rho{\rho'}\big)\Big),\\
u_R&=d+\tau'(\cos\beta,\sin\beta),\\
v_R&=d+\tau'\Big(\cos\big(\beta-2\arccos\frac\tau{\tau'}\big),\sin\big(\beta-2\arccos\frac\tau{\tau'}\big)\Big),\\
w_R&=d+\tau'\Big(\cos\big(\beta+2\arccos\frac\tau{\tau'}\big),\sin\big(\beta+2\arccos\frac\tau{\tau'}\big)\Big);
\end{aligned}
\end{equation*}
then $f(u_S)=u_R$, $v_S,w_S$ are endpoints of the arc $\widetilde A_S(\alpha)$ and $v_R,w_R$ are endpoints of the arc $\widetilde A_R(\beta)$.
It is easy to check that
\[
\widetilde A_S(\alpha)=\{u_S\}\cup\{s\in S:\card\chull_S(\{s,u_S\})=2\},
\]
and
\[
\widetilde A_R(\beta)=\{u_R\}\cup\{r\in R:\card\chull_R(\{r,u_R\})=2\}.
\]
As every collinearity isomorphism obviously preserves collinear hulls, we obtain $f(\widetilde A_S(\alpha))=\widetilde A_R(\beta)$.

Further,
\[
\{v_S,w_S\}=\{s\in S(c,\rho'):\card\chull_S(\{s,u_S\})=3\}
\]
and
\[
\{v_R,w_R\}=\{r\in S(d,\tau'):\card\chull_R(\{r,u_R\})=3\},
\]
and so, by~\eqref{eq:outerCircles}, it follows that
\[
f(\{v_S,w_S\})=\{v_R,w_R\}.
\qedhere
\]
\end{proof}

\begin{lemma}
\label{lem:collinearContinuous}
Let $S=S(c,\rho)\cup S(c,\rho')$ and $R=S(d,\tau)\cup S(d,\tau')$ where $c,d\in\R^2$, $\rho<\rho'$ and $\tau<\tau'$.
Suppose that $f\colon S\to R$ is a collinearity isomorphism.
Then the restriction of $f$ to $S(c,\rho')$ is continuous.
\end{lemma}

\begin{proof}
We would like to apply the same argument as in the proof of \Cref{thm:extension} but we must be a little cautious this time.
In \Cref{thm:extension}, we used the fact that the collection of sets $A_S(\alpha_1)\cap A_S(\alpha_2)$, $\alpha_1,\alpha_2\in\R$, is a base of the topology of the outer circle of $S$.
However, this is not necessarily true if we replace arcs $A_S(\cdot)$ by arcs $\widetilde A_S(\cdot)$.
Indeed, each $A_S(\alpha)$ is always contained in some semicircle.
But, if $\frac\rho{\rho'}<2^{-\frac12}$, then none of the (bigger) arcs $\widetilde A_S(\alpha)$ is contained in a semicircle.
In that case, two arcs $\widetilde A_S(\alpha_1)$, $\widetilde A_S(\alpha_2)$ may intersect at two disjoint smaller arcs (lying on the opposite sides of $S(c,\rho')$).
Fortunately, the collection of all finite intersections of the arcs $\widetilde A_S(\alpha)$, $\alpha\in\R$, still forms a base of the topology of $S(c,\rho')$.
By \Cref{lem:doubleArcs}, each set from this base is mapped by $f$ onto $\widetilde A_R(\beta_1)\cap\ldots\cap\widetilde A_R(\beta_j)$ for some $\beta_1,\ldots,\beta_j\in\R$.
In particular, each set from the base is mapped onto an open subset of $R$.
As $f$ is a bijection from $S(c,\rho')$ onto $S(d,\tau')$
by \Cref{lem:doubleArcs}, we obtain that the restriction of $f^{-1}$ to $S(d,\tau')$ is continuous.
By compactness of the two circles, continuity of $f$ on $S(c,\rho')$ follows as well.
\end{proof}

\begin{lemma}
\label{lem:betwVsCollinear}
Let $S=S(c,\rho)\cup S(c,\rho')$ and $R=S(d,\tau)\cup S(d,\tau')$ where $c,d\in\R^2$, $\rho<\rho'$ and $\tau<\tau'$.
Suppose that $f\colon S\to R$ is a collinearity isomorphism.
Then $f$ is a betweenness isomorphism.
\end{lemma}

\begin{proof}
It suffices to show that $f$ preserves betweenness; then we just apply this weaker result on both $f$ and $f^{-1}$.
So suppose that $x,y,z\in S$ are such that $y$ lies between $x$ and $z$; we want to prove that $f(y)$ lies between $f(x)$ and $f(z)$.
This is trivial if $y\in\{x,z\}$, so we may assume that $y\in(x,z)_S$; then, in particular, $y\in S(c,\rho)$.
By \Cref{lem:doubleArcs}, $f(S(c,\rho'))=S(d,\tau')$.
So, as $f$ preserves collinearity, we easily obtain the desired claim if both $x$ and $z$ belong to $S(c,\rho')$.
If this is not the case then
at least one of the points $x,z$ must
still
belong to $S(c,\rho')$ (otherwise $(x,z)_S=\emptyset$),
so we may assume that $x\in S(c,\rho')$ and $z\notin S(c,\rho')$ (the other case is completely analogous).
Then $f(x)\in S(d,\tau')$ and $f(y),f(z)\in S(d,\tau)$.
So, either $f(y)$ is between $f(x)$ and $f(z)$, or $f(z)$ is between $f(x)$ and $f(y)$; we just need to exclude the latter option.

We fix an arbitrary point $w\in S(c,\rho')$ which is not collinear with $x,y,z$; then $f(w)\in S(d,\tau')$ is not collinear with $f(x),f(y),f(z)$.
Let $w_1$ be the unique element of $S(c,\rho')\setminus\{w\}$ which is collinear with $w,y$, and let $w_2$ be the unique element of $S(c,\rho')\setminus\{w_1\}$ which is collinear with $w_1,z$.
Then $f(w_1)$ is the unique element of $S(d,\tau')\setminus\{f(w)\}$ which is collinear with $f(w),f(y)$, and $f(w_2)$ is the unique element of $S(d,\tau')\setminus\{f(w_1)\}$ which is collinear with $f(w_1),f(z)$.
The set $S(c,\rho')\setminus\{w_1,w_2\}$, resp. $S(d,\tau')\setminus\{f(w_1),f(w_2)\}$, has two connected components.
From a picture, one can easily check that
\begin{itemize}
\item $x$ and $w$ belong to the same connected component of $S(c,\rho')\setminus\{w_1,w_2\}$,
\item if $f(y)$ is between $f(x)$ and $f(z)$ then $f(x)$ and $f(w)$ belong to the same connected component of $S(d,\tau')\setminus\{f(w_1),f(w_2)\}$,
\item if $f(z)$ is between $f(x)$ and $f(y)$ then $f(x)$ and $f(w)$ belong to distinct connected components of $S(d,\tau')\setminus\{f(w_1),f(w_2)\}$.
\end{itemize}
By \Cref{lem:collinearContinuous}, the restriction of $f$ to $S(c,\rho')$ is continuous, and so each connected component of $S(c,\rho')\setminus\{w_1,w_2\}$ is mapped onto a connected subset of $S(d,\tau')\setminus\{f(w_1),f(w_2)\}$.
Thus, it cannot happen that $f(z)$ is between $f(x)$ and $f(y)$ which completes the proof.
\end{proof}

\begin{corollary}
Let $S=S(c,\rho)\cup S(c,\rho')$ and $R=S(d,\tau)\cup S(d,\tau')$ where $c,d\in\R^2$, $\rho<\rho'$ and $\tau<\tau'$.
Then the sets $S$, $R$ are collinearity isomorphic if and only if
\begin{equation*}
\frac\rho{\rho'}=\frac\tau{\tau'}.
\end{equation*}
\end{corollary}

\begin{proof}
Apply \Cref{lem:betwVsCollinear} and \Cref{cor:betwIso}.
\end{proof}

\begin{corollary}
Let $S,R\in\A$ and suppose that $f\colon S\to R$ is a collinearity isomorphism.
Then $f$ extends to a unique
scaled isometry $\Phi\colon\R^2\to\R^2$.
\end{corollary}

\begin{proof}
Apply \Cref{lem:betwVsCollinear} and \Cref{thm:extension}.
\end{proof}


\section{A little about the family \(\An\)}\label{sec:AAn}

Fix $S\in\An$ and let $S=S(c,\rho)\cup S(c',\rho')$, where $c,c'\in\R^2$ and $c\neq c'$. Then exactly one of the five conditions holds (cf.\ also \Cref{fig:An}):
\begin{enumerate}
\item[\rm (a)] one of the circles is strictly inside the other one; i.e.
$|c-c'|<\max\{\rho,\rho'\}-\min\{\rho,\rho'\}$,
\item[\rm (b)] the circles are internally tangent; i.e. 
$|c-c'|=\max\{\rho,\rho'\}-\min\{\rho,\rho'\}$,
\item[\rm (c)] the circles intersect in exactly two points; i.e. 
$\max\{\rho,\rho'\}-\min\{\rho,\rho'\}<|c-c'|<\rho+\rho'$,
\item[\rm (d)] the circles are externally tangent; i.e. $|c-c'|=\rho+\rho'$,
\item[\rm (e)] the circles are strictly outside each other; i.e. $|c-c'|>\rho+\rho'$.
\end{enumerate}

The following observation shows that betweenness isomorphisms distinguish the 
five subfamilies of the family $\An$ discerned by the above conditions.

\begin{remark}\label{rem:An}
\begin{enumerate}[\bgroup\rm(i)\egroup]
\item If two sets from the family $\An$ are betweenness isomorphic, then both must meet the same condition out of the five listed above.
\item None of the sets from the family $\An$ that meet one of the conditions $(b)-(e)$ is betweenness isomorphic with a set belonging to the family $\A$. 
\end{enumerate}
\end{remark}

\begin{proof}
(i) For every $j\in\{a,b,c,d,e\}$, let $\An^j$ denote the subfamily of $\An$ consisting of all sets that meet condition $(j)$. 
Fix $S,R\in\An$ and assume that $f\colon S\to R$ is a betweenness isomorphism. 
Our proof will be divided into five cases.

Case $S\in\An^b$. 
Fix $x\in\ext(S)$ such that for every $y\in\ext(S)\setminus\{x\}$ 
we have $\cardOfInterval (x,y)_S=1$ (such a point exists and it is unique; see \Cref{fig:An}b). 
From \cite[Lemma 2.1(i)]{DKM2024} we conclude that for every $z\in\ext(R)\setminus\{f(x)\}$ we have $\cardOfInterval (f(x),z)_R=1$, which can happen only for $R\in\An^b$, see again \Cref{fig:An}.
So far we have proved that
$S\in \An^b$ if and only if $R\in \An^b$.

\begin{figure}
\centering
\colorlet{myred}{red!85!black}
a)\begin{tikzpicture}[x=4cm,y=4cm,scale=0.4]
\clip(-1.05,-1.15) rectangle (1.15,1.05);
\draw[color=myred, line width=1.5] (0,0) circle (1);
\draw [line width=0.5pt] (0.2,0) circle (0.5);
\draw [line width=0.5pt] (-0.865,-0.5) -- (0.865,0.5);
\draw [color=myred, fill=myred] (0.865,0.5) circle (3.5pt);
\draw [color=myred, fill=myred] (-0.865,-0.5) circle (3.5pt);
\draw [fill=black] (0.575,0.33) circle (2.5pt);
\draw [fill=black] (-0.275,-0.16) circle (2.5pt);
\begin{scriptsize}
\draw[color=black] (0.75,0.5) node {$x$};	
\draw[color=black] (-0.83,-0.4) node {$y$};
\draw[color=black] (0.56,0.24) node {$u$};
\draw[color=black] (-0.17,-0.19) node {$w$};
\end{scriptsize}
\end{tikzpicture}
b)\begin{tikzpicture}[x=4cm,y=4cm,scale=0.4]
\clip (-1.05,-1.15) rectangle (1.2,1.05);
\draw [color=myred, line width=1.5] (0,0) circle (1);
\draw [line width=0.5pt] (0.5,0) circle (0.5);
\draw [line width=0.5pt] (0,-1) -- (1,0);
\draw [color=myred, fill=myred] (1,0) circle (3.5pt );
\draw [fill=black] (0.5,-0.5) circle (2.5pt);
\draw [color=myred, fill=myred] (0,-1) circle (3.5pt);		
\begin{scriptsize}
\draw[color=black] (0.9,0) node {$x$};	
\draw[color=black] (0.5,-0.4) node {$w$};
\draw[color=black] (0,-0.9) node {$y$};
\end{scriptsize}
\end{tikzpicture}
c)\begin{tikzpicture}[x=4cm,y=4cm,scale=0.4]		
\clip (-1.05,-1.15) rectangle (1.35,1.05);
\draw [line width=0.5pt] (0,0) circle (1);
\draw [line width=0.5pt] (0.8,0) circle (0.5);
\draw [color=myred, fill=myred] (0.55,0.84) circle (3.5pt);	
\draw [color=myred, fill=myred] (0.55,-0.84) circle (3.5pt);	
\draw [color=myred, fill=myred] (1.15,0.35) circle (3.5pt);	
\draw [color=myred, fill=myred] (1.15,-0.35) circle (3.5pt);
\draw [shift={(0,0)},line width=1.5pt,color=myred] plot [domain=1:5.3,variable=\t]({1*1*cos(\t r)+0*1*sin(\t r)},{0*1*cos(\t r)+1*1*sin(\t r)});
\draw [shift={(0.75,0)},line width=1.5pt,color=myred] plot [domain=5.5:7.05,variable=\t]({1.24*0.45*cos(\t r)+0*0.45*sin(\t r)},{0*1*cos(\t r)+1*0.51*sin(\t r)});
\end{tikzpicture}
d)\begin{tikzpicture}[x=4cm,y=4cm,scale=0.4]		
\clip (-1.15,-1.15) rectangle (2.2,1.15);
\draw [line width=0.5pt] (0,0) circle (1);
\draw [line width=0.5pt] (1.5,0) circle (0.5);
\draw [fill=black] (1,0) circle (2.5pt );
\draw [color=myred, fill=myred] (0.27,0.96) circle (3.5pt );	
\draw [color=myred, fill=myred] (0.27,-0.96) circle (3.5pt );	
\draw [color=myred, fill=myred] (1.6,0.485) circle (3.5pt );	
\draw [color=myred, fill=myred] (1.6,-0.485) circle (3.5pt );
\draw [shift={(0,0)},line width=1.5pt,color=myred] plot [domain=1.3:5,variable=\t]({1*1*cos(\t r)+0*1*sin(\t r)},{0*1*cos(\t r)+1*1*sin(\t r)});
\draw [shift={(1.45,0)},line width=1.5pt,color=myred] plot [domain=5:7.57,variable=\t]({1*0.55*cos(\t r)+0*1*sin(\t r)},{0*5*cos(\t r)+0.992*0.51*sin(\t r)});
\draw [line width=0.5pt] (1,0) -- (0,1);
\draw [color=myred, fill=myred] (0,1) circle (3.5pt);
\begin{scriptsize}
\draw[color=black] (0,0.9) node {$x$};	
\draw[color=black] (1.1,0) node {$w$};
\end{scriptsize}
\end{tikzpicture}
e)\begin{tikzpicture}[x=4cm,y=4cm,scale=0.4]		
\clip (-1.15,-1.15) rectangle (2.6,1.15);
\draw [line width=0.5pt] (0,0) circle (1);
\draw [line width=0.5pt] (2,0) circle (0.5);
\draw [color=myred, fill=myred] (0.27,0.96) circle (3.5pt);	
\draw [color=myred, fill=myred] (0.27,-0.96) circle (3.5pt);	
\draw [color=myred, fill=myred] (2.1,0.485) circle (3.5pt);	
\draw [color=myred, fill=myred] (2.1,-0.485) circle (3.5pt);
\draw [line width=0.5pt] (1,0) -- (2.28,0.4);
\draw [fill=black] (1,0) circle (2.5pt);
\draw [fill=black] (1.53,0.17) circle (2.5pt);
\draw [color=myred, fill=myred] (2.28,0.4) circle (3.5pt);	
\draw [shift={(0,0)},line width=1.5pt,color=myred] plot [domain=1.3:5,variable=\t]({1*1*cos(\t r)+0*1*sin(\t r)},{0*1*cos(\t r)+1*1*sin(\t r)});
\draw [shift={(1.95,0)},line width=1.5pt,color=myred] plot [domain=5:7.57,variable=\t]({1*0.55*cos(\t r)+0*1*sin(\t r)},{0*5*cos(\t r)+0.992*0.51*sin(\t r)});
\begin{scriptsize}
\draw[color=black] (2.35,0.48) node {$x$};	
\draw[color=black] (0.9,0) node {$w$};
\end{scriptsize}
\end{tikzpicture}
\caption{Sets from $\An$ that meet the respective conditions (a)--(e)
from Section~\ref{sec:AAn}.	
The arcs marked \emph{red} consist of the extreme points of the sets.
}
\label{fig:An}
\end{figure}
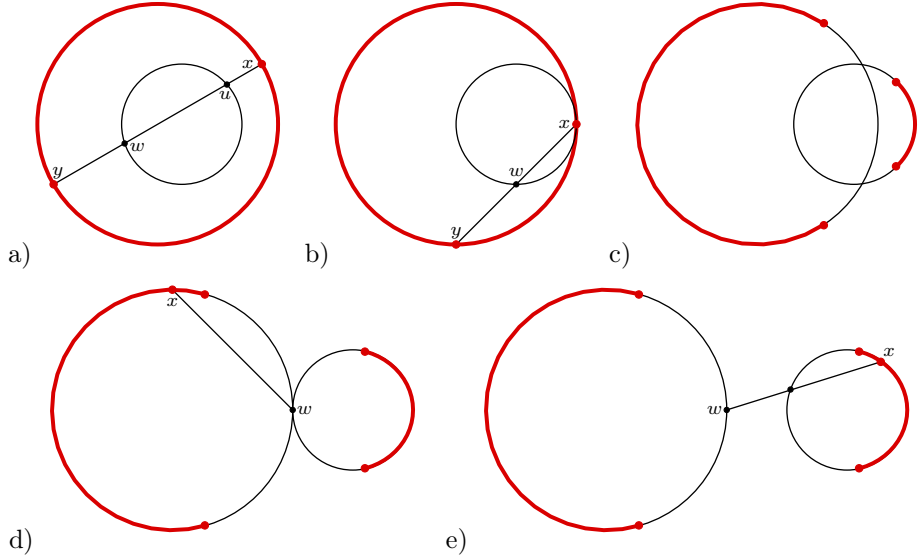

Case $S\in\An^a$. 
In this case we see that for all distinct $u,w\in S\setminus\ext(S)$ there are $x,y\in\ext(S)$ such that $(x,y)_S=\{u,w\}$ (see \Cref{fig:An}a).
By \cite[Lemma 2.1(i)]{DKM2024}, for all distinct
$u',w'\in R\setminus\ext(R)$ there are $x',y'\in\ext(R)$ such that $(x',y')_R=\{u',w'\}$, which is possible only for $R\in\An^a\cup\An^b$
since otherwise this
fails for $u', w'$ lying, e.g., on a suitable
vertical line in \Cref{fig:An}c,d,e.
But from the previous case, we infer that $R\notin\An^b$, and hence $R\in\An^a$.

Case $S\in\An^d$. 
Fix $w\in S\setminus\ext(S)$ such that for every $x\in\ext(S)$ we have $(w,x)_S=\emptyset$ 
(such a point exists and it is unique; see \Cref{fig:An}d). 
\cite[Lemma 2.1(i)]{DKM2024} implies that for every $x'\in \ext(R)$
we have $(f(w),x')_R=\emptyset$, which can take place only for $R\in\An^d$. 

Case $S\in\An^e$.   
In this case for all $w\in S\setminus\ext(S)$ and $x\in\ext(S)$ we have 
$\cardOfInterval (w,x)_S\in\{0,1\}$ (see \Cref{fig:An}e). 
By \cite[Lemma 2.1(i)]{DKM2024} for all $w'\in R\setminus\ext(R)$ and $x'\in\ext(R)$ we have $\cardOfInterval (w',x')_R\in\{0,1\}$, which is true only for $R\in\An^a\cup\An^b\cup\An^e$. But from the previous cases, we conclude that $R\notin\An^a\cup\An^b$, and hence $R\in\An^e$.

Case $S\in\An^c$.
This immediately follows by the previous cases.

(ii) It suffices to note that the proof of (i) will remain valid if we replace the subfamily $\An^a$ with the family $\A$.
\end{proof}

Assertion (i) of \cref{rem:An} allows us to
guess that answering
\cref{Question1} for the 
family $\F=\An$ is a more involved task than that for 
the family $\F=\A$
as in \cref{sec:results}. 
Without further refinements,
the method used to answer \cref{Question1} 
for the family $\F=\A$ is not sufficient to answer \cref{Question1} for the family $\F=\An$. 
However, in view of assertion (ii) of \cref{rem:An}
it is natural to ask if there exists a set belonging 
to the family $\A$ which is not betweenness isomorphic 
with any set of the family $\An$ that meets condition (a). 
The following example shows that this is the case.

\begin{example}\label{ex:AAn}
Fix $c\in\mathbb R^2$, $\rho\in(0,\infty)$, and consider the set $S=S(c,\rho)\cup S(c,2\rho)$ that belongs to the family $\A$. 
We show that $S$ is betweenness isomorphic to
no member of the family $\An$.

Let $aa'e$ be an equilateral triangle inscribed into the outer circle
of~$S$, as in \Cref{fig:concentric}a.
Then the 
sides of the triangle 
 are tangent to the inner circle, meaning that 
$(a,e)_S=\{b\}$, $(a',e)_S=\{b'\}$, and $\cardOfInterval (a,a')_S=1$.
Also,
the invariant $M(S)$ has value $3$
by \Cref{lem:valueOfM}.
Let $d$ and $d'$ be the points on the outer circle of $S$ such that $b\in(a',d)_S$ and $b'\in(a,d')_S$, respectively. 
Then, due to the symmetry of $S$, we have $\cardOfInterval (d,d')_S=1$ (see \Cref{fig:concentric}a).

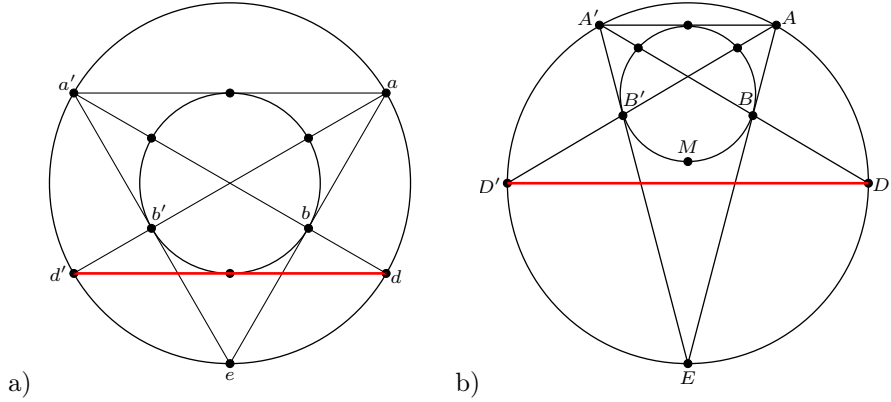
\begin{figure}
\colorlet{myred}{red!85!black}
a)\begin{tikzpicture}[x=4cm,y=4cm,scale=0.6]
\clip(-1.1,-1.15) rectangle (1.2,1.05);
\draw[line width=0.5pt] (0,0) circle (1);
\draw[line width=0.5pt] (0,0) circle (0.5);
\draw [fill=black] (0.865,0.5) circle (2.5pt);
\draw [fill=black] (-0.865,0.5) circle (2.5pt);
\draw [fill=black] (0,0.5) circle (2.5pt);
\draw [fill=black] (0.435,-0.25) circle (2.5pt);
\draw [fill=black] (-0.435,-0.25) circle (2.5pt);
\draw [fill=black] (0.435,0.25) circle (2.5pt);
\draw [fill=black] (-0.435,0.25) circle (2.5pt);
\draw [fill=black] (0.865,-0.5) circle (2.5pt);
\draw [fill=black] (-0.865,-0.5) circle (2.5pt);
\draw [fill=black] (0,-0.5) circle (2.5pt);
\draw [fill=black] (0,-1) circle (2.5pt);
\begin{scriptsize}
\draw[color=black] (0.9,0.55) node {$a$};	
\draw[color=black] (-0.9,0.56) node {$a'$};
\draw[color=black] (0.42,-0.16) node {$b$};
\draw[color=black] (-0.39,-0.16) node {$b'$};
\draw[color=black] (0.92,-0.52) node {$d$};
\draw[color=black] (-0.95,-0.5) node {$d'$};
\draw[color=black] (0,-1.07) node {$e$};
\end{scriptsize}
\draw (30:1) coordinate (A) -- (150:1) coordinate (A') -- (270:1) -- cycle ;
\draw (330:1) coordinate (D) -- (210:1) coordinate (D');
\draw[red, line width=1pt] (210:1)--(330:1);
\draw (A)--(D') (A')--(D);
\end{tikzpicture}
b)\begin{tikzpicture}[x=4cm,y=4cm,scale=0.6]
\clip (-1.15,-1.15) rectangle (1.2,1.05);
\draw [color=black, line width=0.5pt] (0,0) circle (1);
\draw [color=black, line width=0.5pt] (0,0.495) circle (0.375);
\draw [fill=black] (0,0.876307) circle (2.5pt );
\draw [fill=black] (-0.49,0.876307) circle (2.5pt );
\draw [fill=black] (0.49,0.876307) circle (2.5pt );	
\draw [color=black, line width=0.5pt] (-0.49,0.876307) -- (0.49,0.876307);
\draw[color=black, line width=0.5pt] (0,-1) -- (-0.49,0.876307);
\draw[color=black, line width=0.5pt] (0,-1) -- (0.49,0.876307);
\draw [fill=black] (1,0) circle (2.5pt );
\draw [fill=black] (-1,0) circle (2.5pt );
\draw [fill=black] (0,-1) circle (2.5pt );
\draw [fill=black] (-0.36,0.375) circle (2.5pt );
\draw [fill=black] (0.36,0.375) circle (2.5pt );
\draw[color=black, line width=0.5pt] (-0.481754,0.876307) -- (1,0);
\draw[color=black, line width=0.5pt] (0.481754,0.876307) -- (-1,0);
\draw [fill=black] (0.275,0.75) circle (2.5pt );
\draw [fill=black] (-0.275,0.75) circle (2.5pt );
\draw [fill=black] (0,0.12) circle (2.5pt );
\draw[color=red, line width=1pt] (-1,0) -- (1,0);
\begin{scriptsize}
\node [align=left] at (0,-1.08) {$E$};	
\node [align=left] at (0.56,0.92) {$A$};
\node [align=left] at (-0.55,0.92) {$A'$};
\node [align=left] at (1.07,-0.01) {$D$};
\node [align=left] at (-1.1,-0.01) {$D'$};	
\node [align=left] at (0.32,0.47) {$B$};
\node [align=left] at (-0.295,0.47) {$B'$};	
\node [align=left] at (0,0.2) {$M$};	
\end{scriptsize}
\end{tikzpicture}
\caption{A pair of a) concentric circles $S$ with $M(S)=3$
b) non-concentric circles $R$ with `$M(R)=3$'.}
\label{fig:concentric}
\end{figure}

By \cref{rem:An} the set $S$ cannot be betweenness isomorphic to any set belonging to the family $\An$ that meets one of the conditions (b)--(e). 
We want to show that the set $S$ is also not betweenness isomorphic to any set belonging to the family $\An$ that meets condition (a). 
So, let us fix a set $R\in\An$ that meets condition (a). 
By rescaling, translating, and rotating, we can assume that $R=S((0,0),1)\cup S((0,y_0),r)$ with $y_0\in(-1,1)\setminus\{0\}$,
or, if we wish, with $y_0\in(0,1)$, and with
$r\in(0,1-|y_0|)$.  
Now let us suppose that, contrary to our claim, there exists a betweenness isomorphism $f\colon S\to R$.
Using the rotational symmetry of $S$ and, eventually,
reassigning
the positions of the points in 
\Cref{fig:concentric}a
we can assume that $f(a)=A$, $f(a')=A'$ are
positioned on a horizontal line
above $(0,y_0)$,
which is then tangent (by the properties of~$f$) to the inner circle of~$R$,
as in \Cref{fig:concentric}b.
Then, with $f(e)=E$, 
the other two sides of triangle $AA'E$ are tangent
to the inner circle
at $f(b)=B$, $f(b')=B'$, 
as $f$ is a betweenness isomorphism. 
Likewise, the points $f(d)=D$ and $f(d')=D'$ take the position
indicated by \Cref{fig:concentric}b.
Finally, $DD'$ should be tangent to the inner circle,
in contradiction to the figure and the following discussion.

Before we continue our argument, we wish to stop
to make a remark related to the triangle $AA'E$:
By Poncelet's porism (see \cite[Sections 565–567)]{Ponceleti}, 
cf.\ also the survey
\cite{Centina2016i, Centina2016ii}),
there are many triangles
inscribed into $S((0,0),1)$ whose sides
are tangent to $S((0,y_0), r)$.
We may use \eqref{eq:A_R} (with arbitrary $u_R\in R\setminus\ext(R)$)
to make a non-concentric version of definition
 of the arcs ``$A_R(\beta)$''
(here $u_R$ takes the role of the parameter $\beta$; the relation is $u_R=u_R(\beta) = (0 + r\cos \beta, y_0 + r\sin\beta)$ in our case)
and then define $M(\cdot)$ as usual,
which makes $M(\cdot)$ applicable to $R$.
Then the porism may be used to show that $M(R)=3$.
    It also implies that we cannot try to prove our non-isomorphism
    claim building merely on
    (alleged) non-existence of the triangles
    or on the value of the invariant $M(\cdot)$
    as such without further development.

We have
$A=(\sqrt{1-y^2},y)$ for some $y\in(-1,1)\setminus\{\frac{1}{2}\}$;
$y=\frac{1}{2}$
leads to the concentric case. Tedious calculations
give
\[
y_0=1-\sqrt{2-2y},\quad r=y+\sqrt{2-2y}-1,
\]
\[
B=\left(\frac{(y-1)\sqrt{1+y}+\sqrt{2-2y^2}}{\sqrt{2}},\frac{\sqrt{2}y-(1+y)\sqrt{1-y}}{\sqrt{2}}\right);
\]
the second coordinates of points $D$ and $D'$ are the same and equal to
\[
D_y=\frac{2y^2+(2+4y)\sqrt{2-2y}-4y-1}{2\sqrt{2-2y}-5},
\]
whereas the second coordinate of the lowest point $M$
of the inner circle in $R$ equals 
\[
M_y=2-y-2\sqrt{2-2y}.
\]
Then (see \Cref{fig:concentric}b)
\[
M_y-D_y=\frac{
    2\left(\sqrt{1-y}-\sqrt2\right)^2
    \left(\sqrt{1-y}-
          \frac{\sqrt2}{2}
    \right)^2
}{5-2\sqrt{2-2y}}>0;
\] 
equality $M_y=D_y$ can happen only for $y\in\{-1,\frac{1}{2}\}$, but both
these numbers are outside of our domain.
In consequence,
\[
0=\cardOfInterval (D,D')_R=\card f((d,d')_S)=\cardOfInterval (d,d')_S=1,
\] 
a contradiction.
\end{example}

We finish this paper with an example showing that there are sets belonging to the family $\An$ that meet condition (a) which are not betweenness isomorphic with any set of the family $\A$.

\begin{example}
Fix $y\in(-1,\frac{1}{2})\cup(\frac{1}{2},1)$ and consider the set
$R=S((0,0),1)\cup S((0,1-\sqrt{2-2y}),y+\sqrt{2-2y}-1)\in\An$ occurring in \Cref{ex:AAn}. 
We show that $R$
is betweenness isomorphic to no member of $\A$. 
Put $A=(\sqrt{1-y^2},y)$, $A'=(-\sqrt{1-y^2},y)$, and $E=(0,-1)$. 
Then the sides of the triangle $AA'E$ are tangent to the inner circle, i.e.  $(A,E)_R=\{B\}$, $(A',E)_R=\{B'\}$, and $\cardOfInterval (A,A')_R=1$.

Fix now $S=S(c,\rho)\cup S(c,\rho')\in\A$ with $\rho<\rho'$ and suppose, to derive a contradiction, that there exists a betweenness isomorphism $f\colon R\to S$. 
Then the points $A,A',E$ are mapped onto some points $a,a',e$, respectively,
where the sides of the triangle $aa'e$ are tangent to the inner circle 
$S(c,\rho)$.
As we deal with concentric circles, we obtain $\rho'=2\rho$
(see \Cref{fig:concentric}).
Now, \Cref{ex:AAn} reveals that $f$ cannot be a betweenness isomorphism.
\end{example}

\ifRACSAMsubmissionSNJNL
\else
\bibliographystyle{plain}
\fi
\bibliography{bibliography}

\end{document}